\documentclass{amsart}
\usepackage{amsmath,amssymb,amsthm}

\usepackage{microtype}

\usepackage{mathrsfs}
\usepackage{datetime}
\usepackage{enumitem}
\usepackage{xfrac}
\usepackage{subcaption}
\usepackage{varwidth}

\usepackage[margin=1in]{geometry}

\usepackage[colorlinks,linkcolor = blue,citecolor = blue]{hyperref}

\usepackage[disable]{todonotes}

\newtheorem{theorem}{Theorem}
\newtheorem{proposition}{Proposition}[section]

\newtheorem{lemma}{Lemma}[section]
\theoremstyle{definition}
\newtheorem{definition}{Definition}[section]

\theoremstyle{definition}

\numberwithin{equation}{section}

\DeclareMathOperator{\R}{\mathbb{R}}

\def\nn{\nonumber}
\def\({\left(\begin{array}{cccccc}}
\def\){\end{array}\right)}

\def\com#1{\quad{\textrm{#1}}\quad}

\def\nn{\nonumber}

\def\({\left(\begin{array}{cccccc}}
\def\){\end{array}\right)}

\def\bes{\begin{eqnarray}}
\def\ees{\end{eqnarray}}

\def\bel{\begin{equation}\label}

\newcommand{\beq}{\begin{equation}}
\newcommand{\eeq}{\end{equation}}
\newcommand{\bea}{\begin{eqnarray}}
\newcommand{\eea}{\end{eqnarray}}
\newcommand{\beann}{\begin{eqnarray*}}
\newcommand{\eeann}{\end{eqnarray*}}

\newcommand{\RR}{\mathbb{R}}

\newcommand{\eps}{\varepsilon}
\newcommand{\Sweak}{\mathcal{S}_{\text{weak}}}

\newcommand{\Nu}{\mathcal{V}}

\newcommand{\bp}{\begin{proof}}
\newcommand{\ep}{\end{proof}}

\usepackage[american]{babel}

\usepackage{import}
\usepackage{float}
\usepackage{xifthen}
\usepackage{pdfpages}
\usepackage{transparent}

\newcommand{\norm}[1]{\left\lVert#1\right\rVert}

\usepackage{xcolor}
\usepackage{comment}
\usepackage[normalem]{ulem}
\usepackage{bbm}
\usepackage{commath}
\usepackage{mathtools}

\usepackage{graphicx,marvosym}
\usepackage{cleveref}
\usepackage{mdframed}

\title[Quantitative stability]{Quantitative weak-BV stability of ``wild'' solutions to compressible Euler equations, with a view towards higher systems }

\author[Geng Chen]{Geng Chen}
 \address[Geng Chen]{\newline Department of Mathematics, \newline University of Kansas, Lawrence, KS 66045, USA}
 \email{gengchen@ku.edu}

 \author[Cooper Faile]{Cooper Faile}
 \address[Cooper Faile]{\newline Department of Mathematics, \newline The University of Texas at Austin, Austin, TX 78712, USA}
 \email{jcfaile@utexas.edu}

 \author[Sam G. Krupa]{Sam G.  Krupa}
\address[Sam G. Krupa]{\newline Département de mathématiques et applications \newline École normale supérieure, Université PSL, CNRS\newline 45 rue d'Ulm - F 75230 PARIS cedex 05 \newline France}
\email{sam.krupa@ens.fr}

 \date{\today}
\subjclass[2020]{Primary 35L65; Secondary 76N15, 35L45, 35B35.}
 \keywords{compressible Euler system, full Euler, gas dynamics, uniqueness, stability, Hölder, relative entropy, conservation law.}

 \thanks{\textbf{Acknowledgment.} 
 The first author is partially supported by National Science Foundation at grant DMS-2306258, and a SQuaRE at the American Institute of Mathematics. 
The second author is partially supported by National Science Foundation grant DMS-1840314.
The work of the third author is funded by the European Union through the project ``Quantitative Stability and Regularity of Large Data for Conservation Laws.'' Views and opinions expressed are however those of the author(s) only and do not necessarily reflect those of the European Union or European Research Executive Agency (REA). Neither the European Union nor the granting authority can be held responsible for them.}

\mdfdefinestyle{MyFrame}{%
    linecolor=black,
    outerlinewidth=2pt,
    innertopmargin=4pt,
    innerbottommargin=4pt,
    innerrightmargin=4pt,
    innerleftmargin=4pt,
        leftmargin = 4pt,
        rightmargin = 4pt
        }
\begin{document}

\begin{abstract}
For hyperbolic systems of conservation laws, including important physical models from continuum mechanics, the question of stability for large data solutions remains a challenging open problem. In recent work (arXiv:2507.23645) the authors introduce a framework for showing Hölder stability of potentially ``wild'' large data solutions, relative to a class of BV solutions, for systems with two conserved quantities. This is referred to as ``weak-BV'' stability. In this paper, we give a short introduction to the methods while applying them to the ``full'' Euler system with three conserved quantities. We discuss applications to future work for higher systems with additional conserved quantities.
 \end{abstract}

\maketitle

\section{Introduction}
In this paper, we consider the H\"older stability of hyperbolic conservation laws in the space of $L^2$. The general system of hyperbolic conservation laws is in the form of
\beq\label{cl0}
u_t+(f(u))_x=0, \qquad u(x,0)=u^0 ,
\eeq
where $u = (u_1,\dots,u_n) \in \mathcal{V} \subset \R^n$ are the unknowns and $u^{0}\colon\mathbb{R}\to\mathcal{V}$ is the initial data. 
By assumption, the state space $\mathcal{V}$ is an open and bounded convex set.
We assume $f = (f_1,\dots,f_n)$ is $C^4$ on $\mathcal{V}$ up to the boundary.  
Moreover, we assume there exists a strictly convex entropy $\eta \in C^3(\bar{\mathcal{V}})$ and entropy-flux $q\in C^3(\bar{\mathcal{V}})$ verifying
\begin{equation}
    q' = \eta'f'
\end{equation}
in $\mathcal{V}$. We restrict ourselves to entropy-weak solutions, i.e. solutions verifying in a distributional sense
\begin{equation}\label{ineq:entropy}
(\eta(u))_t+(q(u))_x\leq0.
\end{equation}


One fundamental example in the form of \eqref{cl0} is the compressible Euler equation, which is widely used to describe the compressible inviscid flow, such as gas dynamics. Compressible Euler equations are among the oldest partial differential equations (PDEs) written down. It  first appeared in published form in Euler's article in 1757 \cite{euler1757}. In Euler's original work, the system of equations consisted of the momentum and continuity equations. An additional equation, which was called the adiabatic condition, was supplied by Pierre-Simon Laplace in 1816 \cite{Laplace1816} 
(for an extended history see Dafermos \cite[p.~XVII - XXX II]{dafermos_big_book}). 

The 1-d compressible Euler equation, in Lagrangian coordinates, is
\begin{align}
  \tau_t-w_x=0,\nn\\
  w_t+p_x=0,\label{cl}\\
  \left(\frac{1}{2}w^2+\mathcal{E}\right)_t+(w\,p)_x=0, \nn
\end{align}
which can be written in the form of hyperbolic conservation laws \eqref{cl0} with $(u_1,u_2,u_3)=(\tau,w,\frac{1}{2}w^2+\mathcal{E})$. In the literature this system is referred to as ``full Euler'' or ``non-isentropic Euler.''
Here $\tau$ is the specific volume (equivalent to $1/\rho$, for a density $\rho$), $p$ is pressure, $w$ is fluid
velocity, and $\mathcal{E}$ is the specific internal energy.  

The system is closed by an additional equation of state. For convenience, we consider the polytropic ideal $\gamma$-law gas,
with equation of state
\beq\label{eos}
  \mathcal{E}=c_v \theta={\frac{p\,\tau}{\gamma-1}} \com{and} p\,\tau=\bar R\,\theta,\quad\hbox{so}\qquad
  p=\bar Ke^{S/c_v}\tau^{-\gamma}.
\eeq 
Here $S$ is the entropy, $\theta$ is the temperature, $\bar R$, $\bar K$, $c_v$ are
positive constants, and  $\gamma>1$ is the adiabatic gas constant. These state variables satisfy the Gibbs relation 
\beq\label{gib}
\theta\, dS=d\mathcal{E}+p\, d\tau.
\eeq

We use the entropy and entropy-flux pair 
\[\eta=-S=c_v\Big((1-\gamma)\ln\tau-\ln \mathcal{E}+\ln\frac{\bar K}{\gamma-1}\Big), \qquad q=0.\]

Remark \eqref{ineq:entropy} is satisfied since the entropy function $S$ increases after passing any moving $1$-shock or $3$-shock, by the Lax entropy condition (see \cite[p. 341-342]{MR1301779} or \cite{MR2917122}).
The solution of \eqref{cl}, \eqref{ineq:entropy} is equivalent to the one coming from the Eulerian coordinates, even in the class of weak solutions (see \cite{wagnergasdynamics}).

Due to the absence of dissipative effects, classical (i.e. $C^1$) solutions cannot be sustained, and generically, shock waves form in finite time, where the conserved variable becomes discontinuous. Breakdown of classical solutions has been long known. For scalar equations, it goes back to Stokes \cite{Stokes01111848} and Riemann \cite{Riemann}. Shocks are physical phenomenon. For example, the ``sonic boom'' created by supersonic aircraft is characterized by rapid discontinuities in pressure. Once a shock wave forms, one must study weak solutions, and the analysis becomes very difficult, especially when there are many or even infinite shock waves. A natural space to study the solution is the space of \emph{bounded variation} or BV solutions, where the amount of oscillation is controlled. Following the timeline, there are currently three major approaches to studying the global well-posedness of BV solutions in one space dimension: Glimm Scheme, $L^1$ theory, and $L^2$ theory.

The first global existence result was given by \cite{MR0194770}, followed by other methods such as the front tracking scheme in the 1990s \cite{Bbook,MR3443431,dafermos_big_book}, and the vanishing viscosity method (artificial viscosity \cite{MR2150387} and from isentropic Navier-Stokes to Euler \cite{2024arXiv240109305C}). The stability theory is more involved, after the seminal $L^1$ theory in the 1990s \cite{BressanC,LiuYang,BLY,Bbook}, a new approach showing the stability in $L^2$ has been established relatively recently \cite{VASSEUR2008323}. Until the work \cite{2025arXiv250723645C}, these results were limited to giving qualitative stability in $L^2$ (see e.g. \cite{MR4487515,CKV2,Cheng_isothermal}).

The current paper is dedicated to giving a short introduction to a new framework developed for showing \emph{quantitative} Hölder-$1/2$ stability in $L^2$ \cite{2025arXiv250723645C}. Our work is based on the method of relative entropy first introduced by Dafermos \cite{MR546634} and DiPerna \cite{MR523630}, and later extended to $a$-contraction theory by Kang and Vasseur \cite{MR3519973}. Due to the presence of shock waves, Hölder-$1/2$ is the optimal rate of stability result one can expect in $L^2$ (see \cite[Section 1.2]{2025arXiv250723645C}). This work considers systems \eqref{cl0} with two conserved quantities, i.e. $n=2$, also called $2\times 2$ systems. More precisely, we show:
\begin{theorem}[$L^2$-Hölder stability for $2\times 2$ conservation laws, a summary \cite{2025arXiv250723645C}]\label{2x2_theorem}
\hfill

Consider a large class of $2\times 2$ conservation laws, including isentropic Euler and the system of shallow water waves. Fix $T, R>0$. Consider a weak solution $v$ arising as a limit of a sequence of front tracking approximations. Consider also a ``wild'' solution $u$, in a class of solutions analogous to $\Sweak$ (see  \eqref{wild_data}, below). We then have 

\begin{equation}\label{2x2_estimate}
  \|u(\cdot,\tau)-v(\cdot,\tau)\|_{L^2((-R,R))}
  \le K\,\sqrt{\;\|u(\cdot,0)-v(\cdot,0)\|_{L^2((-R-s\tau,\,R+s\tau))}\;}
\end{equation}

for all $\tau\in[0,T]$, for a universal constant $K$, and where $s>0$ is the speed of information (see e.g. \eqref{eq:def-info-speed}, below).

For a general $2\times 2$ system, the solution $v$ can be a limit of small-BV front tracking approximations. For isothermal Euler, the solution $v$ may be a limit of front tracking schemes with large but finite BV-norm.
\end{theorem}

This is related to the question of uniqueness of solutions in $L^\infty$ (see Bressan's Open Problem \# 6 \cite[p.~558]{MR4855159}).

 The estimate \eqref{2x2_estimate} has absolutely no dependence on the BV-norm of the weak solution $u$, which may be infinite. This result links the $L^2$ theory with the  $L^1$-Lipschitz quantitative estimates for BV solutions in the $L^1$ theory.

As a consequence of \Cref{2x2_theorem}, for the isothermal system we have uniqueness of certain large $L^\infty$ solutions with initial data which is \emph{everywhere} locally infinite BV, i.e. infinite BV on every open set (see \cite[Corollary 1]{2025arXiv250723645C}). Compare with the other very recent work \cite{2025arXiv250500420B}, which also gives an infinite-BV uniqueness result, for certain small $L^\infty$ solutions for general {$2\times 2$ genuinely nonlinear} systems.

\vspace{.08in}

In this paper, we prove the H\"older stability of the full $3\times 3$ (non-isentropic) Euler system with three equations (see our \Cref{main-theorem}, below). This new result has a direct application to the uniqueness of solutions in a wide space, and also provides a basis for many other future approaches such as the inviscid limit from non-isentropic Navier-Stokes to the Euler equation. 

In the proof of \Cref{2x2_theorem} in the paper \cite{2025arXiv250723645C}, we use the Bressan-Colombo theory \cite{BressanC} as it is instrumental in the analysis of large-BV front tracking schemes (see Colombo-Risebro \cite{ColomboRisebro}). However, the Bressan-Colombo theory is fundamentally limited to the $2\times 2$ systems and furthermore is quite delicate (with e.g. the use of sophisticated homotopy arguments to measure distance in $L^1$). In this paper, we have a view towards higher systems with many conserved quantities. We use the Bressan-Liu-Yang theory \cite{BLY} which is simpler and, moreover, applies to general $n\times n$ systems with $n$ conserved quantities (see \Cref{sec:weight-con}, below). Within this simplified framework, we aim to convey the main ideas and give a concise proof of our \Cref{main-theorem}, quoting some more technical details from other works when necessary.

A key part of proving Hölder stability involves carefully controlling the entropy dissipation at a shock (see \Cref{shift_existence_prop}, below). In related work,  the second author of this present paper has started a new program to study the dissipation at shocks in all wave families and for general $n\times n $ systems \cite{cooper}.



A final point is that the $L^2$ theory generally and \cite{2025arXiv250723645C} in particular only require that the entropy inequality \eqref{ineq:entropy} hold for one single entropy $\eta$. This is a natural assumption as for many physical systems only one entropy exists (see \cite[p.~13-14]{dafermos_big_book} and \cite[p.~54-55]{dafermos_big_book}). This is true for the Euler system \eqref{cl}.

\bigskip

Now we come to state our main results.

We restrict our study to the solutions verifying what the literature often refers to as the ``Strong Trace Property.''
\begin{definition}[Strong Trace Property]\label{defi_trace}
Let $u\in L^\infty(\RR\times\RR^+)$. We say that $u$ verifies the \emph{Strong Trace Property} if for any Lipschitzian curve $t\to X(t)$, there exist two bounded functions $u_-,u_+\in L^\infty(\RR^+)$ such that for any $T>0$
$$
\lim_{n\to\infty}\int_0^T\sup_{y\in(0,1/n)}|u(X(t)+y,t)-u_+(t)|\,dt=\lim_{n\to\infty}\int_0^T\sup_{y\in(-1/n,0)}|u(X(t)+y,t)-u_-(t)|\,dt=0.
$$
\end{definition}

We will consider the following classes of solutions. Fix an open, convex subset $\mathcal{W}$ of $\mathcal{V}$ such that $\bar{\mathcal{W}}\subset\mathcal{V}$, fix any $d \in \mathcal{W}$, and fix a small $\epsilon>0$. Then define
\begin{align} 
    \mathcal{S}_{\textrm{BV},\epsilon}^0\coloneqq\Bigg\{\text{functions } u^0\colon \mathbb{R} \to \mathcal{W} \, \Big| \,
   \|u^0 - d\|_{L^\infty(\mathbb{R})} \leq \epsilon
  \quad \text{and} \quad \|u^0\|_{BV(\mathbb{R})} \leq \epsilon\Bigg\}.\label{small_BV_data}
  \end{align}
  
We also define the following class of potentially very ``wild'' solutions without smallness.
\begin{figure}
    \centering
        \includegraphics[width=\textwidth]{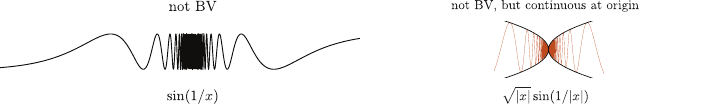}
    \caption{Solutions which behave like $\sin(1/x)$ are outside the BV theory.  However, solutions which behave like $\sqrt{\abs{x}}\sin(1/\abs{x})$ are also outside the BV class, but due to having more controlled oscillations, \emph{verify the Strong Trace Property (\Cref{defi_trace}), live in $\Sweak$, and fit in our $L^2$ theory.}}\label{fig:classes}
\end{figure}
\begin{equation}\label{wild_data}
\mathcal{S}_{\mathrm{weak}}:=
\left\{\,u \in L^\infty(\mathbb{R}\times[0,T); \mathcal{V})\ \big|\
u \text{ is a weak solution to \eqref{cl}, \eqref{ineq:entropy} satisfying \Cref{defi_trace}}\right\},
\end{equation}
for some $T>0$. This is a large class, beyond BV or $\text{BV}_{\text{loc}}$. See \Cref{fig:classes}. Then, we have:

\begin{theorem}[$L^2$-H\"{o}lder stability for full Euler] \label{main-theorem}
Fix $R,\ T > 0$. Then consider $u \in \Sweak$ solving the system~\eqref{cl}. 
Given any initial data $v^0 \in \mathcal{S}_{\textrm{BV},\epsilon}^0$, where $\epsilon > 0$ is sufficiently small, there exists a classical small-BV solution $v$ (via e.g. front tracking). Then for all $\tau \in [0,T]$ we have the bound 
\beq \label{eq:main-est}
    \norm{u(\cdot, \tau) - v(\cdot, \tau)}_{L^2((-R,R))} \leq K \sqrt{ \norm{u(\cdot, 0) - v(\cdot, 0)}_{L^2( (-R -s\tau, R + s\tau))}},
\eeq
for a universal constant $K>0$ and where $s > 0$ is the speed of information (see equation~\eqref{eq:def-info-speed}). 
\end{theorem}



In a nutshell, the idea of the proof is the following: by artificially changing shock speeds, we can keep an approximate front tracking solution (denoted $\psi$, see \Cref{sec:weight-con}) $L^2$-stable with respect to a ``wild'' solution from $\Sweak$. The artificial ``shifting'' of shocks comes with a price, which we keep track of using a dissipation estimate (\Cref{shift_existence_prop}). Taking the limit of the approximate front tracking solutions gives our result.

\vspace{.2cm}
\noindent\textbf{Notation:} $K$ will denote a universal constant, which may change from line to line during computations.


\section{$L^2$ stability for shock, contact and rarefaction} \label{sec:rel-entropy}
Given our entropy, entropy-flux pair $\eta,\ q$ and states $(u,v) \in \mathcal{V} \times \mathcal{V}$ we define the relative quantities 
\beq
    \eta(u|v) := \eta(u) - \eta(v) - \nabla \eta(v)\cdot (u-v), \qquad q(u;v) := q(u) - q(v) - \nabla \eta(v)\cdot (f(u)-f(v)).
\eeq
For fixed $v$ these quantities then constitute an entropy flux pair for the system~\eqref{cl0}.
We recall the following Lemma. 
\begin{lemma}[Lemma 1 of \cite{Leger2011}] \label{lemma:rel-square}
    Let $W$ be a compact subset of $\mathcal{V}$. There exists a constant $C^* > 0$ such that for all $(u,v) \in \mathcal{V}\times W$ we have 
    \beq \frac{1}{C^*} |u-v|^2 \leq \eta(u|v) \leq C^* |u-v|^2.\eeq
\end{lemma}
This Lemma states that controlling integrals of our relative entropy is equivalent to controlling the $L^2$ difference of $u$ and $v$. 
From this, given an $i$-shock $(u_L,u_R,\lambda_i(u_L,u_R))$, where $\lambda_i(u_L,u_R)$ is the shock velocity dictated by the Rankine-Hugoniot condition, and a solution $u\in\Sweak$ we consider the weighted relative entropy
\beq 
    E_t(u; a_1, a_2,h) := \int_{-\infty}^{h(t)} a_1\eta(u| u_L)\,dx + \int_{h(t)}^\infty a_2\eta(u| u_R)\,dx
\eeq
with the constants $a_1,a_2 > 0$ and $h$ a Lipschitz curve. 
Differentiating and using both \Cref{defi_trace} and inequality~\eqref{ineq:entropy} we find the dissipation functional $D_t$,
\beq \label{eq:diss_functional}
\begin{aligned}
    \frac{d}{dt}E_t(u; a_1, a_2,h) \leq & \overbracket{a_2\left[q(u(h(t)+,t);u_R)-\dot{h}(t) \eta(u(h(t)+,t)|u_R)\right]}^{\text{dissipation from right side of $h$}}  \\
    &-\quad\overbracket{a_1\left[q(u(h(t)-,t);u_L)-\dot{h}(t)\eta(u(h(t)-,t)|u_L)\right]}^{\text{dissipation from left side of $h$}} =: D_t(u;a_1,a_2,h).
    \end{aligned}
\eeq
Thus, showing the existence of a triple $(a_1,a_2,h)$ which gives bounds on $D_t$ then gives us bounds on the $L^2$ distance of $u$ and $(u_L,u_R,h)$, which is no longer a shock but rather a shifted front. The forthcoming propositions give us such a triple for the shocks and contact discontinuities of~\eqref{cl}. 

\begin{proposition}{\cite[Proposition 3.1]{2025arXiv250723645C}}\label{shift_existence_prop}
Consider system \eqref{ineq:entropy}, \eqref{cl}.
Let $d\in \Nu$. Then there exist constants $\hat{\lambda},\ K,\ C_1,\ \eps,$ and $ \alpha>0$ such that the following holds.\vskip0.1cm
Consider any shock  $(u_L,u_R)$ with $|u_L-d|+|u_R-d|\leq \eps$. Consider also any $u\in \Sweak$, any $T>0$, $T_{start} \in [0,T]$, and any $x_0\in \RR$.
 Let  $s_0=|u_L-u_R|$ be the size of the shock.
Then for any $a_1,\ a_2 > 0$ satisfying 
\begin{eqnarray}
1- 2C_1 s_0\leq \frac{a_2}{a_1}\leq 1-\frac{C_1s_0}{2} &&\text{ if } (u_L,u_R) \text{ is a 1-shock}\label{con1}\\
1+\frac{C_1s_0}{2}\leq \frac{a_2}{a_1}\leq 1+2C_1s_0 && \text{ if } (u_L,u_R) \text{ is a 3-shock}\label{con2},
\end{eqnarray}
there exists a Lipschitz shift function  $h\colon[T_{start},T]\to\mathbb{R}$, with $h(T_{start})=x_0$, such that our dissipation functional $D_t$ satisfies the bound
\begin{equation} \label{diss:shock}
\frac{d}{dt}E_t(u; a_1, a_2,h) \leq D_t \leq -a_2K s_0 (\dot h(t) - \lambda_i(u_L,u_R))^2
\end{equation}
for almost all  $t\in[T_{start},T]$, where $\lambda_i(u_L,u_R)$ is the Rankine-Hugoniot velocity of the $i$-shock $(u_L,u_R)$, $i = 1\text{ or }3$.

\vskip0.2cm

Moreover, if $(u_L,u_R)$ is a 1-shock, then for almost all  $t\in[T_{start},T]$:  
\beq \label{eq:1-shock-shift-bound}
-\frac{\hat{\lambda}}{2}\leq \dot{h}(t) \leq -\alpha<0.
\eeq
Similarly, if $(u_L,u_R)$ is a 3-shock,  then for almost all  $t\in[T_{start},T]$: 
\beq \label{eq:3-shock-shift-bound}
0< \alpha\leq \dot{h}(t) \leq\frac{\hat{\lambda}}{2}.
\eeq



\end{proposition} 

\vskip0.3cm


At contact discontinuities we experience weighted entropy decay without shift due to \cite[Theorem 4.1]{MR3475284}.

\begin{lemma}{\cite[Theorem 4.1]{MR3475284}}\label{lem:contact} 
Consider a contact discontinuity $(u_L,u_R)$ of the system~\eqref{cl}  located at $x = \beta$ and a solution $u \in \Sweak$. 
Then the associated relative entropy decays without shift when we select our weights to be the absolute temperature at $u_L$ and $u_R$:
\beq
    \frac{d}{dt} E_t(u; \theta_L, \theta_R, \beta)\leq D_t \leq 0.
\eeq

\end{lemma}

We note that this Lemma differs from \Cref{shift_existence_prop} in that it requires the ratio 
\beq\label{con3}
\frac{a_2}{a_1}= \frac{\theta_R}{\theta_L}
\quad\hbox{for a contact},
\eeq
where $a_1,a_2$ are from the context of \Cref{shift_existence_prop}.
This impacts the construction of our space-time weight in \Cref{sec:weight-con}, where we must be careful to preserve this exact equality. 

Finally, for dissipation at the discretized rarefaction fans which arise in front tracking, we have
\begin{proposition}{\cite[Proposition 4.4]{MR4487515}}\label{dissipation_discrete_rarefaction}
There exists a constant $C>0$ such that the following is true.
For any rarefaction wave $\bar{u}(y)$ $v_L\leq y\leq v_R$ of \eqref{cl}, let 
$$
\delta=|v_L-v_R|+\sup_{y\in [v_L,v_R]} |u_L-\bar{u}(y)|, \qquad \bar{u}(v_L)=u_L, \ \bar{u}(v_R)=u_R.
$$
 Then for any $u\in \Sweak$, any $v_L\leq v\leq v_R$,  and any $t>0$ we have the bound
 $$
 \int_{0}^{t} \left\{
q(u(tv+,t);u_R)-q(u(tv-,t);u_L)-v\left( \eta(u(tv+,t)|u_R) -\eta(u(tv-,t)|u_L) \right)
 \right\}\,dt
 \leq C\delta |u_L-u_R|
 t.
 $$
\end{proposition}


\section{The modified front tracking scheme and weight function $a(x,t)$\label{sec:weight-con}}
One key step in the proof of the main theorem is to estimate the difference between a front tracking $\nu$-approximate solution $u(x,t)$ with true Rankine--Hugoniot speeds (up to the $\mathcal{O}(\nu)$ error) and a \emph{shifted} $\nu$-approximate solution $\psi(x,t)$ where, given a fixed $u\in\Sweak$, shocks $(u_L,u_R)$ move according to the function $h$ from \Cref{shift_existence_prop} except for small perturbations to avoid interactions of more than two waves.

Here we use the same definition of front tracking scheme in \cite{CKV2} to define the $\nu$-approximate solutions: both the un-shifted and shifted front tracking schemes.

We note the un-shifted front tracking scheme in \cite{CKV2} is a special case of the $\nu$-approximate front tracking scheme used in  \cite{BLY} (named as ``$\epsilon$-approximate'' front tracking scheme in \cite{BLY}).

In fact, in the front tracking scheme of \cite{CKV2}, there exists a small constant $\nu$, controlling three types of errors, as in \cite{BLY}:
\begin{itemize}
\item Errors in the speeds of shock and rarefaction fronts, to avoid interaction of more than two waves.
\item The maximum strength of any rarefaction front.
\item The total strength of all non-physical waves.
\end{itemize}

In the front tracking scheme, after discretizing the initial data to a piecewise-constant approximation, we let initial waves to propagate until the next interaction. And we define two Riemann solvers (accurate solver and simplified solver) to treat each wave interaction. In these two solvers, we use four types of wave fronts: shock, rarefaction jump, contact discontinuity and non-physical wave. Thanks to the decay of the Glimm functional in \Cref{prop:delta} below, the total number of wave interactions is finite and the BV norm of approximate solutions (shifted or un-shifted) has a uniform bound. So one can prove the convergence of the scheme to a BV solution using Helly's theorem. We refer the reader to \cite{CKV2} and \cite{Bbook} for details.

For any $\nu$-approximate (shifted or un-shifted) solution $u$, we use standard notions $L$ and $Q$ for total variation and Glimm potential at any time $t$:
\beq\label{Ldef}
L(u)\coloneqq \hbox{TV}(u)(t),\qquad\quad 
Q(u)\coloneqq\sum_{i,j:\hbox{approaching waves}}|\sigma_i| |\sigma_j|(t),
\eeq
where the summation runs over all pairs of approaching waves, with strengths $|\sigma_i|$ and $|\sigma_j|$, at time $t$ (see \cite[p.~5]{BLY} for details). 

\begin{proposition}{\cite{CKV2}}\label{prop:delta}
There exists $\kappa>0$ such that for any $\epsilon$ small enough, where the initial total variation of the (shifted, or un-shifted) $\nu$-approximate solution $u$ is less than $\epsilon$, then the functional $\Upsilon(u)\coloneqq L(u)+\kappa Q(u)$ is decreasing in time. 
\end{proposition}

\bigskip

We need to construct a space-time weight so that when we calculate the dissipation of entropy for $\psi$ at shocks and contact discontinuities, we have the proper weighting of the left and right states at each discontinuity (e.g. $a_1,a_2$ in \Cref{shift_existence_prop} or $\theta_L,\theta_R$ in \Cref{lem:contact}). See \Cref{sec:diss_calc}, below, for details.

\textbf{Space-time weight.} Now we revisit the weight function $a(x,t)$ in \cite{CKV2} and will use it for the current paper.



First, we give our definition of $\bar\sigma$. In
particular,
for a 2-wave, we define 
\beq\label{2waves}
\bar\sigma=\frac{1}{C_1J}\Delta \theta,
\eeq where $\Delta g=g_R-g_L$ measures the difference of $g$ on two sides of waves, the constant $C_1$ is defined in Proposition \ref{shift_existence_prop} and
$J=\frac{1}{2}\inf_{\bar{\mathcal{W}}}  \theta>0$
where the temperature $\theta$ defined in \eqref{eos} always has a positive lower bound.
This infimum includes all possible states our approximation $u_\nu$ can achieve when $\epsilon$ in~\eqref{small_BV_data} is chosen sufficiently small. 

For any 1-wave and 3-wave, we choose $\bar\sigma=\pm|\Delta u|$, where $\bar\sigma$ takes negative sign on a shock and positive sign on a rarefaction front. For non-physical wave, we choose $\bar\sigma=|\Delta u|$. Finally, at each discontinuity $x_i(t)$, we label the corresponding $\bar\sigma$ as  $\bar\sigma^{(i)}$.

We introduce the following measure $\mu(\cdot,t)$ as a sum of Dirac measures in $x$:
\[\mu(x,t)= -\sum_{i: 1\text{-shock}} |\bar\sigma^{(i)}| \delta_{\{x_i(t)\}}+\sum_{i: 3\text{-shock}} |\bar\sigma^{(i)} |\delta_{\{x_i(t)\}}+\sum_{i: 2\text{-wave}}\frac{J\hat a}{\hat \theta}\bar\sigma^{(i)}\delta_{\{x_i(t)\}}
\]
where $\hat a$ and $\hat \theta$ denote the $a$ and $\theta$ values to the left of the contact. Then 
\beq\label{a_def}
a(x,t)=1+ C_1\Big(L(t)+\kappa Q(t)+\int_{-\infty}^x \mu(x,t)\,dx\Big),
\eeq
where the constant $C_1$ is defined in Proposition \ref{shift_existence_prop}.
The weight $a$ is constant across the two sides of a rarefaction or a non-physical wave.
\footnote{More precisely, we need to define $a$ inductively from the extreme left. Assume $a(\cdot,t)$ consists of $n$ constant states $a^{(1)}$ to $a^{(n)}$ from left to right. Define
$
a^{(1)}(t)=1+ C_1(L(t)+\kappa Q(t)),$ then
$a^{(i+1)}(t)=a^{(i)}(t)+C_1(-(\bar\sigma_1^{(i)})_-+(\bar\sigma_3^{(i)})_-+\frac{J\hat a}{\hat \theta}\bar\sigma_2^{(i)}),$
where there are at most three waves $\bar\sigma_1^{(i)},\bar\sigma_2^{(i)}, \bar\sigma_3^{(i)}$ at the point when $a$ changes from $a^{(i)}$ to
$a^{(i+1)}$.
Here $g_-=-\min(g,0)$ for any function $g$, and $\hat a$ and $\hat \theta$ denote the $a$ and $\theta$ values to the left of the contact, in particular, between outgoing 1- and 2-waves right after an interaction.
}

By \cite{CKV2}, we know the weight $a(x,t)$ defined in \eqref{a_def} satisfies \eqref{con1}, \eqref{con2}, \eqref{con3} {at shocks and contact discontinuities}, respectively, 
and the time decay 
\begin{equation}\label{eqa2}
a(x,t+)\leq a(x,t-),
\end{equation}
except at any point of interaction.




\section{$L^1$ distance between solutions with real and shifted shock speeds and error estimate}
In this section, we will give some key estimate for the weighted $L^1$ distance, proposed by \cite{BLY},
between a front tracking $\nu$-approximate solution $u(x,t)$ with true Rankine--Hugoniot speeds (up to the $\mathcal{O}(\nu)$ error) and a \emph{shifted} $\nu$-approximate solution $\psi(x,t)$ whose shocks are moved by prescribed shifts (following \Cref{shift_existence_prop}). Both $u$ and $\psi$ take initial values in $\mathcal{S}_{\textrm{BV},\epsilon}^0$ defined in \eqref{small_BV_data}.



We adopt the notation in \cite{Bbook} to define the weighted distance for hyperbolic conservation laws \eqref{cl0} in the general $n\times n$ case, which includes the Euler equation as a special example ($n=3$). 

First, by $\sigma\mapsto S_i(\sigma)(u_0)\equiv S_i(u_0)$, we denote the parametrized $i$-shock curve through the point $u_0$, with $S_i(0)(u_0)=u_0$. If the i-th characteristic field is linearly degenerate (e.g. 2-nd family of Euler equation), the curve $S_i$ will be parametrized by arc-length. If the i-th field is genuinely non-linear (1-st and 3-rd families for Euler equations), the parametrization is chosen as follows. We choose the right eigenvector $r_i(u)$ corresponding to the eigenvalue $\lambda_i(u)$ of $Df(u)$, so that $r_i\cdot \nabla \lambda_i\equiv 1$. Moreover, we choose the parametrization of the $i$-shock curve $\sigma\mapsto S_i(\sigma)(u_0)$ and $i$-rarefaction curve $\sigma\mapsto R_i(\sigma)(u_0)$ so that
\[
\frac{d}{d\sigma}\lambda_i(S_i(\sigma)(u_0))\equiv1,\qquad
\frac{d}{d\sigma}\lambda_i(R_i(\sigma)(u_0))\equiv1,
\qquad
\lambda_i(S_i(\sigma)(u_0))=\lambda_i(R_i(\sigma)(u_0))=\lambda_i(u_0)+\sigma.
\]
Same as before, $S_i(\sigma)(u_0)$ represents the entropy shock curve when $\sigma<0$.

For any two small BV functions $v$ and $u$, we define the scalar functions $q_i$ implicitly by:
\[
v(x) = S_n(q_n(x)) \circ \cdots \circ S_1(q_1(x))(u(x)).
\]
Then
\beq\label{l1}
\frac{1}{K} |v(x) - u(x)| \leq \sum_{i=1}^n |q_i(x)| \leq K |v(x) - u(x)|
\eeq
for some $K>0$.

Then we define the weighted distance as:
\[
\Phi(u, v) := \sum_{i=1}^n \int_{-\infty}^{\infty} |q_i(x)| W_i(x) \, dx
\]
where the weights $W_i$ are defined as:
\[
W_i(x) := 1 + \kappa_1 A_i(x) + \kappa_2 \left[ Q(u) + Q(v) \right].
\]
Here $Q$ is the interaction potential defined in \eqref{Ldef}, and we
denote $\mathcal{J}(u)$, $\mathcal{J}(v)$ as the sets of all jumps in $u$ and in $v$ respectively and $\mathcal{J}=\mathcal{J}(u)\cup \mathcal{J}(v)$. For any wave with strength $|\sigma_\alpha|$ located at $x_\alpha$, we use $k_\alpha$ to denote its family. Then 
\[
A_i(x) := \left[ \sum_{\alpha \in\mathcal{J}, x_\alpha < x, i < k_\alpha\leq n} + \sum_{\alpha \in\mathcal{J}, x_\alpha > x, 1\leq k_\alpha<i} \right] |\sigma_\alpha|
\]
if the $i$-th field is linearly degenerate. For the general genuinely nonlinear field, $A_i(x)$ includes an additional term
\begin{eqnarray}
\left[ \sum_{k_\alpha=i, \alpha \in\mathcal{J}(u), x_\alpha < x} + \sum_{k_\alpha=i, \alpha \in\mathcal{J}(v), x_\alpha > x}\right] |\sigma_\alpha| &\hbox{if}\quad q_i(x)<0\nonumber\\
\left[ \sum_{k_\alpha=i, \alpha \in\mathcal{J}(v), x_\alpha < x} + \sum_{k_\alpha=i, \alpha \in\mathcal{J}(u), x_\alpha > x} \right] |\sigma_\alpha|&\hbox{if}\quad q_i(x)\geq0.\nonumber
\end{eqnarray}
When $\epsilon$ is small enough, $1\leq W_i\leq 2$, so by \eqref{l1}, we know
\beq \label{eq:L1-equiv}
\frac{1}{K}\cdot\|v-u\|_{L^1}\leq\Phi(u,v)\leq 2K\cdot\|v-u\|_{L^1}.
\eeq

We now give an important estimate for our theory.

\begin{lemma}[$L^1$ estimate for shifted front tracking]\label{lem:shifted-ft}
Let $u$ be a $\nu$-approximate front tracking solution without shifts and and let $\psi$ be a modified $\nu$-approximate front tracking solution whose shocks follow shift functions $h_\alpha(t)$ according to \Cref{shift_existence_prop} (see \Cref{sec:weight-con}, above). Then, in the sense of distributions,
\begin{equation}\label{eq:phi-ineq}
\frac{d}{dt}\, 
\Phi\big(u(\cdot,t),\psi(\cdot,t)\big)
\;\le\;
K \sum_{\alpha} |\psi(h_\alpha(t)+,t)-\psi(h_\alpha(t)-,t)|\,\big|\dot h_\alpha(t)- \dot{h}_{\mathrm{true},\alpha}(t)\big| +\mathcal{O}(\nu).
\end{equation}
Here the sum is over the shocks of $\psi$ at time $t$, $K$ is a universal constant, $\dot h_\alpha$ is the imposed (shifted) shock speed in $\psi$, and $\dot h_{\mathrm{true},\alpha}$ is the Rankine--Hugoniot speed the same discontinuity would have in the un-shifted scheme. 
\end{lemma}

\begin{proof}



Now we consider the distance $\Phi(u, \psi)$, with $\nu$-approximate solution $u$
 and shifted $\nu$-approximate solution $\psi$.
 
 Since there are in total only finitely many waves in $u$ and $\psi$, there are in total only finitely many interaction times when two waves in $u$ interact or two waves in $\psi$ interact.
 If one wave in $u$ attaches to a wave in $\psi$ for a while, we can perturb the speed of the wave in $u$ a bit to avoid it, since we allow $\nu$ error in the speeds of waves in $u$. 

\smallskip
 \paragraph{\bf Case 1.}
Now we first consider one of finitely many open time intervals $(t_j,t_{j+1})$, in which there are no interactions between any two waves in  $u$ or in $\psi$.

For any $t\in(t_j,t_{j+1})$,  all jumps in  $u$ and $\psi$ are denoted as $\{x_\alpha=x_\alpha(t),$ $\alpha\in\mathcal{J}=\mathcal{J}(u)\cup \mathcal{J}(\psi)\}$
where we use $J(u)$ and $J(\psi)$ to denote all jumps in $u$ and $\psi$, respectively.

Then it is easy to get 
\beq\label{Phi1}
\frac{d}{dt}\Phi(u(\cdot,t),\psi(\cdot,t))=\sum_{\alpha\in\mathcal{J}}\sum_{i=1}^n\{|q_i(x_\alpha-)|W_i(x_\alpha-)-|q_i(x_\alpha+)|W_i(x_\alpha+) \}\dot{x}_\alpha
\eeq
(see \cite{BLY} or \cite{Bbook} equation (8.16)).

On the other hand, starting from the profile of $\psi(\cdot,t^*)$ for any fixed $t^*\in (t_j,t_{j+1})$, we consider the corresponding $\nu$-approximate solution $\bar\psi$ with initial data $\psi(\cdot,t^*)$, in a short open time interval including $t^*$. The function $\bar\psi$ will have  wave speeds $\dot {\bar x}_\alpha$, including almost correct Rankine-Hugoniot for shocks,  allowing up to $\nu$ error. We remark that $\dot {\bar x}_\alpha$ may only be different from  $\dot { x}_\alpha$ for shock waves.  We still have 
\begin{eqnarray}
\frac{d}{dt}\Phi(u(\cdot,t),\bar\psi(\cdot,t))&=&
\sum_{\alpha\in\mathcal{J}(u)}\sum_{i=1}^n\{|q_i( x_\alpha-)|W_i(x_\alpha-)-|q_i(x_\alpha+)|W_i( x_\alpha+) \}\dot{x}_\alpha \label{Phi2}\\
&&
+\sum_{\alpha\in\mathcal{J}(\bar\psi)}\sum_{i=1}^n\{|q_i(\bar x_\alpha-)|W_i(\bar x_\alpha-)-|q_i(\bar x_\alpha+)|W_i(\bar x_\alpha+) \}\dot{\bar x}_\alpha,\nonumber
\end{eqnarray}
where we use $J(\bar\psi)$ to denote all jumps in $\bar\psi$.

By the result in \cite[equation (3.7)]{BLY}  or \cite[equation (8.22)]{Bbook},
\[
\frac{d}{dt}\Phi(u(\cdot,t), \bar\psi(\cdot,t))=\mathcal{O}(\nu).
\]

Note, at $t=t^*$, $\bar x_\alpha=x_\alpha$, and $\bar \psi$ and $\psi$ share the same profile, so  by \eqref{Phi1} and \eqref{Phi2},
\begin{eqnarray*}
&&\frac{d}{dt}\Phi(u(\cdot,t),\psi(\cdot,t))|_{t=t^*}
-\frac{d}{dt}\Phi(u(\cdot,t),\bar\psi(\cdot,t))|_{t=t^*}\\
&=&\Bigg(\sum_{\alpha\in\mathcal{J}(\psi)}\sum_{i=1}^n\{|q_i(x_\alpha-)|W_i(x_\alpha-)-|q_i(x_\alpha+)|W_i(x_\alpha+) \}(\dot{x}_\alpha-\dot{\bar x}_\alpha)\Bigg)\Big|_{t=t^*} .
\end{eqnarray*}
{Now we prove the right hand side of this equation is less than
\beq\label{error_est}
K\sum_{\alpha} |\psi(x_\alpha+)-\psi(x_\alpha-)||\dot{x}_\alpha-\dot{\bar x}_\alpha|\Big|_{t=t^*}+\mathcal{O}(\nu),
\eeq
with the sum is over shocks of $\psi$ at time $t=t^*$.}


To prove \eqref{error_est},
we use that fact that, for the genuinely nonlinear or linearly degenerate family
\[
W_i(x_\alpha+)-W_i(x_\alpha-)=\mathcal{O}(1)\sigma_\alpha\quad \hbox{if}\quad q_i(x_\alpha+)q_i(x_\alpha-)>0;
\]
\beq\label{q_con_1}
q_i(x_\alpha-)=\mathcal{O}(1)|\sigma_\alpha|\quad \hbox{if}\quad q_i(x_\alpha+)q_i(x_\alpha-)\leq 0;
\eeq
and
\beq\label{q_con_2}
q_i(x_\alpha+)-q_i(x_\alpha-)=\mathcal{O}(1)|\sigma_\alpha|
\eeq
(see page 162 of \cite{Bbook}).
Then by \eqref{q_con_1} and \eqref{q_con_2}, we know
\[
q_i(x_\alpha+)=\mathcal{O}(1)|\sigma_\alpha|\quad \hbox{if}\quad q_i(x_\alpha+)q_i(x_\alpha-)\leq 0.
\]
 Then 
\begin{align*}
&|q_i(x_\alpha-)|W_i(x_\alpha-)-|q_i(x_\alpha+)|W_i(x_\alpha+) \\
=&-W_i(x_\alpha+) (|q_i(x_\alpha+)|-|q_i(x_\alpha-)|)
-(W_i(x_\alpha+)-W_i(x_\alpha-))|q_i(x_\alpha-)|\\
=&\,\mathcal{O}(1) |\sigma_\alpha|.
\end{align*}
{Hence \eqref{error_est} holds, using the equivalence between $|\sigma_\alpha|$ and $|\psi(x_\alpha+)-\psi(x_\alpha-)|$ and $\dot{x}_\alpha=\dot{\bar x}_\alpha$ for corresponding rarefactions and non-physical shocks in $\psi$ and $\bar\psi$.}

On the other hand, we already know from \cite{Bbook,BLY} that 
\[
\frac{d}{dt}\Phi(u(\cdot,t),\bar\psi(\cdot,t))|_{t=t^*}=\mathcal{O}(\nu).
\]
{Since $t^*$ is arbitrary, we have the estimate
\beq\label{key est}
\frac{d}{dt}\Phi(u(\cdot,t),\psi(\cdot,t))\leq K\sum_{\alpha} |\psi(x_\alpha+,t)-\psi(x_\alpha-,t)||\dot{x}_\alpha-\dot{\bar x}_\alpha|+\mathcal{O}(\nu),
\eeq 
with the sum over shocks of $\psi$ at time $t\in(t_j,t_{j+1})$.}

\bigskip

 \paragraph{\bf Case 2.}
When $t=t_j$ with an interaction, it is easy to see 
\[
\Phi(u(\cdot,t+),\psi(\cdot,t+))\leq \Phi(u(\cdot,t-),\psi(\cdot,t-)).
\]
In fact, we still find the  $\nu$-approximate solution $\bar \psi$ in a short open time interval including $t_j$,  starting from the profile of $ \psi(\cdot,t_j)$.
Then it is clear that 
\[\Phi(u(\cdot,t+),\psi(\cdot,t+))=\Phi(u(\cdot,t+),\bar \psi(\cdot,t+))\leq \Phi(u(\cdot,t-),\bar \psi(\cdot,t-))=\Phi(u(\cdot,t-),\psi(\cdot,t-)).\]
Here using the Lax entropy condition, when we trace back the solution in time at any interaction between two waves in $\psi$, we know incoming waves of $\psi$ and $\bar\psi$ share the same configuration (see \cite{2025arXiv250723645C}).
\end{proof}

\section{Proof of \Cref{main-theorem}}
Here we present the proof of our Main Theorem. For brevity and clarity, we leave out some very technical details and refer the reader to \cite{CKV2}, \cite[Section 6]{2025arXiv250723645C}, and \cite[Section 6]{GiesselmannKrupa2025}. 

\subsection{Step 1: Dissipation computation}\label{sec:diss_calc}
Given a fixed $v^0 \in \mathcal{S}_{\textrm{BV},\epsilon}^0$ for $\epsilon > 0$ sufficiently small there exists a sequence of front tracking approximations $v_\nu$ from \cite[Chapter 7]{Bbook} such that $v_\nu(\cdot,t)\to v(\cdot,t)$ in $L^2$ for all $t$, including $v_\nu(\cdot, 0) \to v^0$ in $L^2$ as $\nu\to0$. 
Fixing our parameter $\nu$, let $\psi$ be a shifted front tracking solution as described in \cite{CKV2} such that $\psi(\cdot, 0) = v_\nu(\cdot, 0)$.
Associated to this shifted approximation $\psi$, we also fix a weight $a(x,t)$ as defined in \Cref{sec:weight-con}.

Given consecutive interaction times of waves $t_j < t_{j+1}$ for our shifted solution $\psi$ we define $h_1,\dots,h_N\colon [t_j, t_{j+1}] \to \R$ to be the positions of the discontinuities of $\psi$, where $h_i < h_{i+1}$ for all $i$. 
Furthermore, since we are restricting ourselves to the information cone we are able to define 
\beq h_0(t) \coloneqq -R+ s(t-\tau),\ h_{N+1}(t) \coloneqq R - s(t-\tau) \eeq
where $s > 0$ is our speed of information, taken sufficiently large to verify
\beq \label{eq:def-info-speed}
    |q(a;b)| \leq s \eta(a| b)
\eeq
for all states $a \in \mathcal{V}$ and states $b$ attained by our shifted approximation $\psi$. 
We further suppose that $s$ is larger than $\hat \lambda$, which results in $h_0,\ h_{N+1}$ moving faster than any artificial shift constructed by \Cref{shift_existence_prop}, rarefaction front, or non-physical front. 

For $t \in [t_j,t_{j+1}]$, in any fixed quadrilateral
\beq Q = \{ (x,r)\, | \, t_j < r < t,\ h_i(r) < x < h_{i+1}(r) \} \eeq
the functions $\psi|_Q$ and $a|_Q$ are both constant. 
Integrating~\eqref{ineq:entropy} over $Q$ with entropy $\eta\left(\,\cdot\,\middle|\psi|_Q\right) $, entropy-flux $q\left(\,\cdot\,;\psi|_Q\right)$, and using that $u$ has the Strong Trace Property (\Cref{defi_trace}) we find\footnote{We remark that to simplify presentation of this proof we are ignoring a technical issue of stopping and restarting the clock, which requires use of approximate limits in this computation. See \cite[Section 7]{MR4487515} for discussion of this technicality.}
\beq \label{eq:single-quad}
\begin{aligned}
    \int_{h_i(t)}^{h_{i+1}(t)} a(x, t-) \eta(u(x,t)| \psi(x,t))\,dx \leq& \int_{h_i(t_j)}^{h_{i+1}(t_j)} a(x, t_j+) \eta(u(x,t_j)| \psi(x,t_j))\,dx \\
    & \hspace{1in}+ \int_{t_j}^t F_i^+(r) - F_{i+1}^-(r)\,dr, 
\end{aligned}\eeq
where the various $F^+_i,\ F^-_i$ are the right/left parts of the dissipation functional ~\eqref{eq:diss_functional} for the fronts bounding the left/right edges of $Q$. 
Summing all quadrilaterals left-to-right between $t_j, t$ we can now pair the traces at each front (matching the left- and right-sided dissipation, as in \eqref{eq:diss_functional}), and by virtue of the bounds in \Cref{sec:rel-entropy} we find 
\beq \label{eq:sum-between-interactions-SUB}
\begin{aligned}
    \int_{-R - s(t-\tau)}^{R + s(t-\tau)} a(x, t-) \eta(u(x,t)| \psi(x,t))\,dx
    \leq& \int_{h_0(t_j)}^{h_{N+1}(t_j)} a(x, t_j+) \eta(u(x,t_j)| \psi(x,t_j))\,dx \\
    &- \frac{1}{K}\int_{t_j}^t \sum_{i\in \mathcal{S}(r)} | \psi(h_i(r) +,r) - \psi(h_i(r) -, r)|( \lambda^s_i - \dot h_i(r) )^2 \,dr \\
    &\hspace{1.7in}+ K(t-t_j)\left[ \nu + \sum_{i \in \mathcal{NP}(t)} |\sigma_i|\right].
\end{aligned}
\eeq
where $\mathcal{S}(t),\ \mathcal{NP}(t) \subset \{ 1,\dots, N\}$ correspond to the fronts $h_i$ of shocks and non-physical fronts of $\psi$ at time $t$, respectively, and $ \lambda^s_i = \lambda_{\alpha_i}(\psi(h_i(r)+,r),\psi(h_i(r)-,r)) $ is the Rankine-Hugoniot velocity of the $i$-{th} shock, which is of family $\alpha_i \in \{1,3\}$. 
We remark that $\mathcal{S}(t),\ \mathcal{NP}(t)$ are constant between times of interaction and that the errors caused by perturbing wave velocities (to prevent more than two waves from interacting at once) and discretized rarefactions are bounded by $K \nu(t-t_j)$ in the above expression.

Fixing $\tau \in [0,T]$ we now sum~\eqref{eq:sum-between-interactions-SUB} over each time interval $[t_j,t_{j+1}]$ where $t_1 <\cdots < t_{J-1}$ are consecutive times of front interaction prior to $\tau$ and $t_{J-1}\leq t_J\coloneqq \tau$.
Using the time decay of our weight $a$~\eqref{eqa2} and the convexity of $\eta$, we receive the bound
\beq \label{eq:diss-comp}\begin{aligned}
     \int_{-R}^{R} a(x, \tau) \eta(u(x,\tau)| \psi(x,\tau))\,dx \leq&  \int_{R + s\tau}^{R - s\tau} a(x, 0) \eta(u(x,0)| \psi(x,0))\,dx \\
     &- \frac{1}{K} \int_0^\tau \sum_{i \in \mathcal{S}(r)} | \psi(h_i(r) +,r) - \psi(h_i(r) -, r)| ( \lambda^s_i - \dot h_i(r) )^2 \,dr\\
     &\hspace{1.9in}+ K \left[ \nu + \sup_{t \in [0,\tau]}\sum_{i \in \mathcal{NP}(t)} |\sigma_i | \right],
\end{aligned}\eeq
where the $\nu$ term contains all errors due to the discretization of rarefaction fans and perturbation of wave speeds to avoid multiple simultaneous interactions in $\psi$. 

\subsection{Step 2: Bound on shifting} 

Recalling that $v_\nu(\cdot,0) = \psi(\cdot,0)$ and using \Cref{lem:shifted-ft}, we have 
\begin{align*}
    \Phi(v_\nu(\cdot, \tau), \psi(\cdot, \tau)) \leq& K\int_0^\tau \sum_{i \in \mathcal{S}(r)} | \psi(h_i(r) -,r) - \psi(h_i(r) +, r) |  \left| \lambda^s_i -\dot h_i(r) \right| + \mathcal{O}(\nu), 
\end{align*}
where $\dot h_i$ is the shifted velocity of the $i$-th shock front in $\psi$. 
The Cauchy-Schwartz inequality gives us 
\beq\label{eq:L1-CS}\begin{aligned}
    \Phi(v_\nu(\cdot, \tau), \psi(\cdot, \tau)) \leq& K\left[ \int_0^\tau \sum_{\mathclap{i \in \mathcal{S}(r)}}| \psi(h_i(r) -,r) - \psi(h_i(r) +, r) |\,dr \right]^{1/2} \\
    &\times \left[ \int_0^\tau \sum_{\mathclap{i \in \mathcal{S}(r)}} | \psi(h_i(r) -,r) - \psi(h_i(r) +, r) | \left|  \lambda^s_i -\dot h_i(r) \right|^2 \,dr \right]^{1/2} \hspace{-1em}+ \mathcal{O}(\nu) \\ 
    \leq& K\sqrt{\epsilon} \left[ \int_{-R - s\tau}^{R + s\tau} a(x,0)\eta(u(x,0) | \psi(x,0)) \,dx  + K \left[ \nu + \sup_{t \in [0,\tau]}\sum_{i \in \mathcal{NP}(t)} |\sigma_i|\right] \right]^{1/2} \hspace{-1em} + \mathcal{O}(\nu)
\end{aligned}\eeq
where the last bound follows by \eqref{eq:diss-comp} and the uniform bound on $\norm{\psi(\cdot,t)}_{BV(\R)}$ from \Cref{prop:delta}.

\subsection{Step 3: Putting it together}
The triangle inequality gives us 
\beq
    \norm{u(\cdot, \tau) - v_\nu(\cdot, \tau)}_{L^1( (-R,R)) } \leq \norm{u(\cdot, \tau) - \psi(\cdot, \tau)}_{L^1( (-R,R)) } + \norm{\psi(\cdot, \tau) - v_\nu(\cdot, \tau)}_{L^1( (-R,R)) }.
\eeq
Using the inclusion $L^2 \hookrightarrow L^1$ for finite measure spaces we find 
\[
    \norm{u(\cdot, \tau) - v_\nu(\cdot, \tau)}_{L^1( (-R,R)) } \leq \sqrt{2R}\norm{u(\cdot, \tau) - \psi(\cdot, \tau)}_{L^2( (-R,R)) } + \norm{\psi(\cdot, \tau) - v_\nu(\cdot, \tau)}_{L^1( (-R,R)) }.
\]
We further note that by $v_\nu(\cdot,0)=\psi(\cdot,0)$, the equivalence~\eqref{eq:L1-equiv}, dissipation estimate~\eqref{eq:diss-comp}, \Cref{lemma:rel-square}, and the bound~\eqref{eq:L1-CS} that
\begin{multline*}
\norm{u(\cdot, \tau) - v_\nu(\cdot, \tau)}_{L^1( (-R,R)) } \leq K \left[ \norm{u(\cdot, 0) - v_\nu(\cdot, 0)}_{L^2( (-R -s\tau, R + s\tau))}^2 + K \left[ \nu + \sup_{t \in [0,\tau]}\sum_{i \in \mathcal{NP}(t)} |\sigma_i|\right] \right]^{1/2} 
\\ + \mathcal{O}({\nu}).
\end{multline*}
Finally, recalling that $\norm{f}_{L^2} \leq \sqrt{ \norm{f}_L^1 \norm{f}_{L^\infty}}$ and taking the limit $\nu \to 0$ we arrive at 
\beq 
\norm{u(\cdot, \tau) - v(\cdot, \tau)}_{L^2( (-R,R)) } \leq K \sqrt{\norm{u(\cdot, 0) - v(\cdot, 0)}_{L^2( (-R -s\tau, R + s\tau)} },
\eeq
where we note the term summing the strength of non-physical fronts converges to zero (see \cite[Lemma 3.1]{baiti1998front}).


\bibliographystyle{apalike}
\bibliography{references-2}

\begin{center} 
\includegraphics[width=.4\linewidth]{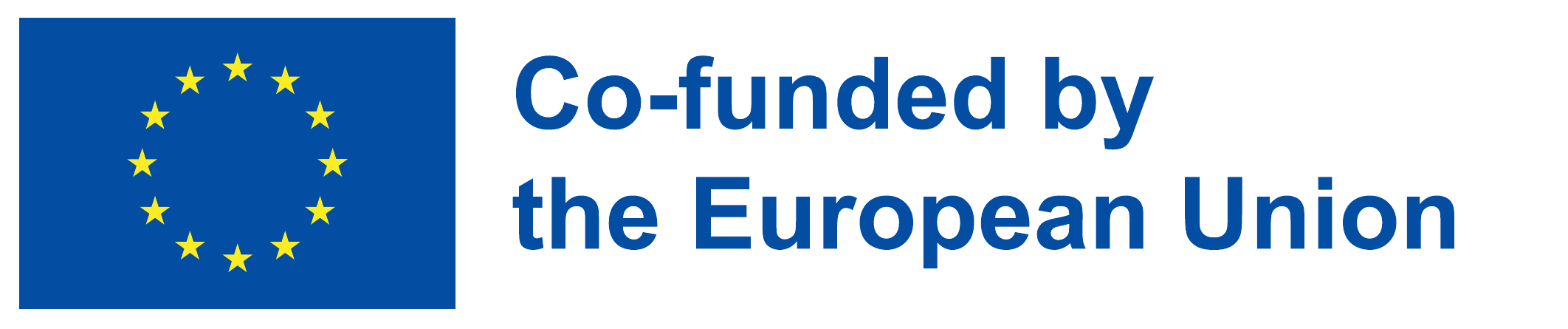}
\end{center}

\end{document}